\documentclass{article}
\usepackage{amsfonts}
\usepackage{amsmath}
\usepackage{harvard}

\newtheorem{theorem}{Theorem}

\newtheorem{corollary}[theorem]{Corollary}

\newtheorem{lemma}[theorem]{Lemma}

\newenvironment{proof}[1][Proof]{\noindent\textbf{#1.} }{\ \rule{0.5em}{0.5em}}

\input{tcilatex}

\begin{document}

\title{On the Determinants and Inverses of Circulant Matrices with Pell and
Pell-Lucas Numbers}
\author{Durmu\c{s} Bozkurt\thanks{%
e-mail: dbozkurt@selcuk.edu.tr} \& Fatih Y\i lmaz\thanks{%
e-mail: fyilmaz@selcuk.edu.tr} \and Department of Mathematics, Science
Faculty, \and Sel\c{c}uk University, 42075 Kampus, Konya, Turkey}
\maketitle

\begin{abstract}
Let $\mathbb{P}=circ(P_{1},P_{2},\ldots ,P_{n})$ and $\mathbb{Q}%
=circ(Q_{1},Q_{2},\ldots ,Q_{n})\ $be $n\times n$ circulant matrices where $%
P_{n}$ and $Q_{n}$ are $n$th Pell and Pell-Lucas numbers, respectively. The
determinants of the matrices $\mathbb{P}$\ and $\mathbb{Q}$\ were expressed
by the Pell and Pell-Lucas numbers. After, we prove that the matrices $%
\mathbb{P}$\ and $\mathbb{Q}$\ are the invertible for $n\geq 3$ and then the
inverses of the matrices $\mathbb{P}$ and $\mathbb{Q}$\ are derived.
\end{abstract}

\section{Introduction}

\bigskip Circulant matrices have a widely range of application in signal
processing, coding theory, image processing, digital image disposal, design
of self-regress and so on. Also, numerical solutions of the certain types of
elliptic and parabolic partial differential equations with periodic boundary
conditions often involve linear systems $Cx=b$ with $C$ a circulant matrix
[9-11].

The $n\times n$ circulant matrix $C_{n}=circ(c_{0},c_{1},\ldots ,c_{n-1}),$\
associated with the numbers $c_{0},c_{1},\ldots ,c_{n-1},$\ is defined by%
\begin{equation}
C_{n}:=\left[ 
\begin{array}{ccccc}
c_{0} & c_{1} & \ldots & c_{n-2} & c_{n-1} \\ 
c_{n-1} & c_{0} & \ldots & c_{n-3} & c_{n-2} \\ 
\vdots & \vdots & \ddots & \vdots & \vdots \\ 
c_{2} & c_{3} & \ldots & c_{0} & c_{1} \\ 
c_{1} & c_{2} & \ldots & c_{n-1} & c_{0}%
\end{array}%
\right] .  \label{1}
\end{equation}

\bigskip Circulant matrices have a {wide} range of {applications, for
examples} in signal processing, coding theory, image processing, digital
image disposal, self-regress {\ design} and so on. {Numerical} solutions of
the certain types of elliptic and parabolic partial differential equations
with periodic boundary conditions often involve linear systems {associated
with circulant matrices} [9,10].

\bigskip The eigenvalues and eigenvectors of $C_{n}$ are well-known [4,14]:%
\begin{equation*}
\lambda _{j}=\dsum\limits_{k=0}^{n-1}c_{k}\omega ^{jk},\ \ \ \ \
j=0,1,\ldots ,n-1,
\end{equation*}%
where $\omega :=\exp (\frac{2\pi i}{n})$ and $i:=\sqrt{-1}$ and the
corresponding eigenvectors%
\begin{equation*}
x_{j}=(1,\omega ^{j},\omega ^{2j},\ldots ,\omega ^{(n-1)j}),\ \ \
j=0,1,\ldots ,n-1.
\end{equation*}

Thus we have the determinants and inverses of nonsingular circulant matrices
[1,3,4,14]:%
\begin{equation*}
\det (C_{n})=\dprod\limits_{j=0}^{n-1}(\dsum\limits_{k=0}^{n-1}c_{k}\omega
^{jk})
\end{equation*}%
and%
\begin{equation*}
C_{n}^{-1}=circ(a_{0},a_{1},\ldots ,a_{n-1})
\end{equation*}%
where $a_{k}=\frac{1}{n}\tsum\nolimits_{j=0}^{n-1}\lambda _{j}\omega ^{-jk}$ 
$,$ and $k=0,1,\ldots ,n-1)~[4].$ When $n$ is getting large,the above the
determinant and inverse formulas are not very handy to use. If there is some
structure among $c_{0},c_{1},\ldots ,c_{n-1},$\ we may be able to get more
explicit forms of the eigenvalues , determinants and inverses of $C_{n}.$
Recently, studies on the circulant matrices involving interesting number
sequences appeared. In [1] the determinants and inverses of the\linebreak
circulant matrices $A_{n}=circ(F_{1},F_{2},\ldots ,F_{n})$ and $%
B_{n}=circ(L_{1},L_{2},\ldots ,L_{n})$\ are derived where $F_{n}$ and $L_{n}$%
\ are $n$th Fibonacci and Lucas numbers, respectively. In [2] the $r$%
-circulant matrix is defined and its norms was computed. The norms of
Toeplitz matrices [??] involving Fibonacci and Lucas numbers are obtained
[5]. Miladinovic and Stanimirovic [6] gave an explicit formula of the
Moore-Penrose inverse of singular generalized Fibonacci matrix. Lee and et
al. found the factorizations and eigenvalues of Fibonacci and symmetric
Fibonacci matrices [7].

\bigskip The Pell and Pell-Lucas sequences $\{P_{n}\}$ and $\{Q_{n}\}$ are
defined by $P_{n}=2P_{n-1}+P_{n-2}$ and $Q_{n}=2Q_{n-1}+Q_{n-2}$\ with
initial conditions $P_{0}=0$ and $P_{1}=1$ and $Q_{0}=2$ and $Q_{1}=2$\ for $%
n\geq 2,$ respectively. The Pell numbers are used to approximate the square
root of two, to find square triangular numbers, to construct integer
approximations to the right isosceles triangle, and to solve certain
combinatorial enumeration problems. Let $\mathbb{P}=circ(P_{1},P_{2},\ldots
,P_{n})$\ and $\mathbb{Q}=circ(Q_{1},Q_{2},\ldots ,Q_{n}).$\ The \ aim of
this paper is to establish some useful formulas for the determinants and
inverses of $\mathbb{P}$ and $\mathbb{Q}$\ using the nice properties of the
Pell and Pell-Lucas numbers. Matrix decompositions are derived for $\mathbb{P%
}$ and $\mathbb{Q}$\ in order to obtain the results.

\section{Determinants of circulant matrices with the\protect\linebreak Pell
and Pell-Lucas numbers}

Recall that $\mathbb{P}=circ(P_{1},P_{2},\ldots ,P_{n})$ and $\mathbb{Q}%
=circ(Q_{0},Q_{1},\ldots ,Q_{n-1}),$ i.e. where $P_{k}$ and $Q_{k}$\ are $k$%
th Pell and Pell-Lucas numbers, respectively, with the recurence relations $%
P_{k}=2P_{k-1}+P_{k-2},$ $Q_{k}=2Q_{k-1}+Q_{k-2},$\ the initial conditions $%
P_{0}=0,$ $P_{1}=1,$ $Q_{0}=2$ and $Q_{1}=2$\ ($k\geq 2).$\ Let $\alpha $\
and $\beta $\ be the roots of $x^{2}-2x-1=0.$ Using the Binet formulas [11,
p. 142] for the sequences $\{P_{n}\}$ and $\{Q_{n}\},$\ one has%
\begin{equation}
P_{n}=\frac{\alpha ^{n}-\beta ^{n}}{2\sqrt{2}}  \label{2}
\end{equation}%
and%
\begin{equation}
Q_{n}=\alpha ^{n}+\beta ^{n}\text{.}  \label{3}
\end{equation}%
\ 

\begin{theorem}
Let \bigskip $n\geq 3.$ Then%
\begin{equation}
\det (\mathbb{P})=(P_{1}-P_{n+1})^{n-2}(P_{1}-2P_{n})+\dsum%
\limits_{k=2}^{n-1}\left[ P_{k-1}P_{n}^{n-k}(P_{1}-P_{n+1})^{k-2}\right]
\label{4}
\end{equation}%
where $P_{k}$ is $k$th Pell number.
\end{theorem}

\begin{proof}
Obviously, $\det (\mathbb{P})=104$ for $n=3.$ It satisfies (\ref{4}). For $%
n>3,$ we selec the matrices $M$ and $N$ so that when we multiply $\mathbb{P}$%
\ with $M$ on the left and $N$ on the right we obtain a special Hessenberg
matrix that have nonzero entries only on first two rows, main diagonal and
subdiagonal: \ 
\begin{equation}
M:=\left[ 
\begin{array}{rrrrrrr}
1 & 0 & 0 & 0 & \ldots & 0 & 0 \\ 
-2 & 0 & 0 & 0 & \ldots & 0 & 1 \\ 
-1 & 0 & 0 & 0 & \ldots & 1 & -2 \\ 
0 & 0 & 0 & 0 & \ldots & -2 & -1 \\ 
\vdots & \vdots & \vdots & \vdots & \ddots & \vdots & \vdots \\ 
0 & 0 & 1 & -2 & \ldots & 0 & 0 \\ 
0 & 1 & -2 & -1 & \ldots & 0 & 0%
\end{array}%
\right]  \label{5}
\end{equation}%
and%
\begin{equation*}
N:=\left[ 
\begin{array}{cccccc}
1 & 0 & 0 & \ldots & 0 & 0 \\ 
0 & \left( \frac{P_{n}}{P_{1}-P_{n+1}}\right) ^{n-2} & 0 & \ldots & 0 & 1 \\ 
0 & \left( \frac{P_{n}}{P_{1}-P_{n+1}}\right) ^{n-3} & 0 & \ldots & 1 & 0 \\ 
0 & \left( \frac{P_{n}}{P_{1}-P_{n+1}}\right) ^{n-4} & 0 & \ldots & 0 & 0 \\ 
\vdots & \vdots & \vdots &  & \vdots & \vdots \\ 
0 & \left( \frac{P_{n}}{P_{1}-P_{n+1}}\right) & 1 & \ldots & 0 & 0 \\ 
0 & 1 & 0 & \ldots & 0 & 0%
\end{array}%
\right] .
\end{equation*}%
Notice that we obtain the following equivalence:%
\begin{eqnarray*}
S &=&M\mathbb{P}N \\
&=&\left[ 
\begin{array}{ccccccc}
1 & g_{n}^{\prime } & P_{n-1} & P_{n-2} & P_{n-3} & \ldots & P_{2} \\ 
& g_{n} & P_{n-2} & P_{n-3} & P_{n-4} & \ldots & P_{1} \\ 
&  & P_{1}-P_{n+1} &  &  &  &  \\ 
&  & -P_{n} & P_{1}-P_{n+1} &  &  & 0 \\ 
&  &  & -P_{n} & P_{1}-P_{n+1} &  &  \\ 
&  &  &  &  &  &  \\ 
& 0 &  &  & -P_{n} & P_{1}-P_{n+1} &  \\ 
&  &  &  &  & -P_{n} & P_{1}-P_{n+1}%
\end{array}%
\right]
\end{eqnarray*}%
and $S$ is Hessenberg matrix, where 
\begin{equation*}
g_{n}^{\prime }:=P_{n}+\dsum\limits_{k=2}^{n-1}P_{k}\left( \frac{P_{n}}{%
1-P_{n+1}}\right) ^{n-k},
\end{equation*}%
\begin{equation*}
g_{n}:=P_{1}-2P_{n}+\dsum\limits_{k=2}^{n-1}P_{k-1}\left( \frac{P_{n}}{%
P_{1}-P_{n+1}}\right) ^{n-k}.
\end{equation*}
Then we have%
\begin{equation*}
\det (S)=\det (M)\det (\mathbb{P}_{n})\det (N)=(P_{1}-P_{n+1})^{n-2}g_{n}.
\end{equation*}%
Since%
\begin{equation*}
\det (M)=\det (N)=\left\{ 
\begin{array}{l}
\ \ 1,\ n\equiv 1\ or\ 2\ \func{mod}4 \\ 
-1,\ n\equiv 0\ or\ 3\ \func{mod}4%
\end{array}%
\right.
\end{equation*}%
for all $n>3,$%
\begin{equation*}
\det (M)\det (N)=1
\end{equation*}%
and (\ref{4}) follows.
\end{proof}

\begin{theorem}
\bigskip \bigskip Let \bigskip $n\geq 3.$ Then%
\begin{equation}
\det (\mathbb{Q})=2(2-Q_{n+1})^{n-2}(2-3Q_{n})+2\dsum\limits_{k=2}^{n-1}%
\left[ (Q_{k+1}-3Q_{k})(Q_{n}-2)^{n-k}(2-Q_{n+1})^{k-2}\right] .  \label{6}
\end{equation}%
where $Q_{k}$ is $k$th Pell-Lucas number.
\end{theorem}

\begin{proof}
Since $n\geq 3,$ $\det (\mathbb{Q})=2464$ for $n=3,$ satisfies (\ref{6}).
For $n>3,$ we select the matrices $K$ and $L$ so that when we multiply $%
\mathbb{Q}$\ with $K$ on the left and $M$ on the right we obtain a special
Hessenberg matrix that have nonzero entries only on the first two rows, main
diagonal and subdiagonal: 
\begin{equation}
K:=\left[ 
\begin{array}{rrrrrrr}
1 & 0 & 0 & 0 & \ldots & 0 & 0 \\ 
-3 & 0 & 0 & 0 & \ldots & 0 & 1 \\ 
-1 & 0 & 0 & 0 & \ldots & 1 & -2 \\ 
0 & 0 & 0 & 0 & \ldots & -2 & -1 \\ 
\vdots & \vdots & \vdots & \vdots & \ddots & \vdots & \vdots \\ 
0 & 0 & 1 & -2 & \ldots & 0 & 0 \\ 
0 & 1 & -2 & -1 & \ldots & 0 & 0%
\end{array}%
\right]  \label{7}
\end{equation}%
and%
\begin{equation*}
L:=\left[ 
\begin{array}{cccccc}
1 & 0 & 0 & \ldots & 0 & 0 \\ 
0 & \left( \frac{-2+Q_{n}}{Q1-Q_{n+1}}\right) ^{n-2} & 0 & \ldots & 0 & 1 \\ 
0 & \left( \frac{-2+Q_{n}}{Q1-Q_{n+1}}\right) ^{n-3} & 0 & \ldots & 1 & 0 \\ 
0 & \left( \frac{-2+Q_{n}}{Q1-Q_{n+1}}\right) ^{n-4} & 0 & \ldots & 0 & 0 \\ 
\vdots & \vdots & \vdots &  & \vdots & \vdots \\ 
0 & \left( \frac{-2+Q_{n}}{Q1-Q_{n+1}}\right) & 1 & \ldots & 0 & 0 \\ 
0 & 1 & 0 & \ldots & 0 & 0%
\end{array}%
\right] .
\end{equation*}%
We have%
\begin{eqnarray*}
U &=&K\mathbb{Q}L \\
&=&\left[ 
\begin{array}{ccccccc}
Q_{1} & u_{n}^{\prime } & Q_{n-1} & Q_{n-2} & \ldots & Q_{3} & Q_{2} \\ 
& u_{n} & Q_{n}-3Q_{n-1} & Q_{n-1}-3Q_{n-2} & \ldots & Q_{4}-3Q_{3} & 
Q_{3}-3Q_{2} \\ 
&  & Q_{1}-Q_{n+1} &  &  &  &  \\ 
&  & 2-Q_{n} & Q_{1}-Q_{n+1} &  &  &  \\ 
&  &  & 2-Q_{n} &  &  &  \\ 
&  &  &  & \ddots &  &  \\ 
&  &  &  & \ddots & Q_{1}-Q_{n+1} &  \\ 
&  &  &  &  & 2-Q_{n} & Q_{1}-Q_{n+1}%
\end{array}%
\right]
\end{eqnarray*}%
and U is Hessenberg matrix, where%
\begin{equation*}
u_{n}:=Q_{1}-3Q_{n}+\dsum\limits_{k=2}^{n-1}\left( Q_{k+1}-3Q_{k}\right)
\left( \frac{Q_{n}-2}{Q_{1}-Q_{n+1}}\right) ^{n-k},
\end{equation*}%
\linebreak 
\begin{equation*}
u_{n}^{\prime }:=\dsum\limits_{k=2}^{n}Q_{k}\left( \frac{Q_{n}-2}{%
Q_{1}-Q_{n+1}}\right) ^{n-k}.
\end{equation*}%
Then we obtain%
\begin{equation*}
\det (U)=\det (K)\det (\mathbb{Q})\det (L)=Q_{1}(Q_{1}-Q_{n+1})^{n-2}u_{n}.
\end{equation*}%
Since%
\begin{equation*}
\det (K)=\det (L)=\left\{ 
\begin{array}{l}
\ \ 1,\ n\equiv 1\ or\ 2\ \func{mod}4 \\ 
-1,\ n\equiv 0\ or\ 3\ \func{mod}4%
\end{array}%
\right.
\end{equation*}%
for all $n>3,$%
\begin{equation*}
\det (K)\det (L)=1
\end{equation*}
and we have (\ref{6}).
\end{proof}

\section{\protect\bigskip Inverses of $\mathbb{P}$ and $\mathbb{Q}$}

We will use the well-known fact that the inverse of a nonsingular circulant
matrix is also circulant [14, p.84] [12, p.33] [4, p.90-91].

\begin{theorem}
\bigskip Let $\mathbb{P}=circ(P_{1},P_{2},\ldots ,P_{n})$\ is invertible
when $n\geq 3$.
\end{theorem}

\begin{proof}
From Theorem 1, we have $\det (\mathbb{P})=104\neq 0$ and $\det (\mathbb{P}%
)=-18560\neq 0$ for $n=3$ and $n=4,$ respectively. Then $\mathbb{P}$ is the
invertible for $n=3,4$. Let $n\geq 5.$ TheBinet formula for Pell numbers
gives $P_{n}=\frac{\alpha ^{n}-\beta ^{n}}{2\sqrt{2}}$ where $\alpha +\beta
=2,\alpha \beta =-1$ and $\alpha -\beta =2\sqrt{2}.$ Then we have%
\begin{eqnarray*}
u(\omega ^{k}) &=&\dsum\limits_{r=1}^{n}P_{r}\omega
^{kr-k}=\dsum\limits_{r=1}^{n}\left( \frac{\alpha ^{r}-\beta ^{r}}{2\sqrt{2}}%
\right) \omega ^{kr-k}=\frac{1}{2\sqrt{2}}\dsum\limits_{r=1}^{n}\left(
\alpha ^{r}-\beta ^{r}\right) \omega ^{kr-k} \\
&=&\frac{1}{2\sqrt{2}}\left[ \frac{\alpha (1-\alpha ^{n})}{1-\alpha \omega
^{k}}-\frac{\beta (1-\beta ^{n})}{1-\beta \omega ^{k}}\right] ,\ \ (1-\alpha
\omega ^{k},1-\beta \omega ^{k}\neq 0) \\
&=&\frac{1}{2\sqrt{2}}\left( \frac{(\alpha -\beta )-(\alpha ^{n+1}-\beta
^{n+1})+\alpha \beta \omega ^{k}(\alpha ^{n}-\beta ^{n})}{1-\alpha \omega
^{k}-\beta \omega ^{k}+\alpha \beta \omega ^{2k}}\right) \\
&=&\frac{1-P_{n+1}-P_{n}\omega ^{k}}{1-2\omega ^{k}-\omega ^{2k}},\ \
k=1,2,\ldots ,n-1.
\end{eqnarray*}%
If there existed $\omega ^{k}$ $(k=1,2,\ldots ,n-1)$ such that $u(\omega
^{k})=0,$ then we would have $1-P_{n+1}-P_{n}\omega ^{k}=0$ for $1-2\omega
^{k}-\omega ^{2k}\neq 0.$ Hence $\omega ^{k}=\frac{1-P_{n+1}}{P_{n}}.$ It is
well known that%
\begin{equation}
\omega ^{k}=\exp \left( \frac{2k\pi i}{n}\right) =\cos \left( \frac{2k\pi }{n%
}\right) +i\sin \left( \frac{2k\pi }{n}\right)  \label{8}
\end{equation}%
where $i:=\sqrt{-1}.$ Since $\omega ^{k}=\frac{1-P_{n+1}}{P_{n}}$ is a real
number, $\sin \left( \frac{2k\pi }{n}\right) =0$ so that $\omega ^{k}=-1$
for $0<\frac{2k\pi }{n}<2\pi .$ However $x=-1$ is not a root of the equation 
$1-P_{n+1}-P_{n}x=0$ $(n\geq 5),$ a contradiction. i.e. $u(\omega ^{k})\neq
0 $ for any $\omega ^{k}$ where $k=1,2,\ldots ,n-1$ and $n\geq 5.$ Thus, the
proof is completed by [1, Lemma 1.1].
\end{proof}

\begin{lemma}
Let $C=\bigskip (c_{ij})$ be an $(n-2)\times (n-2)$ matrix defined by%
\begin{equation*}
c_{ij}=\left\{ 
\begin{array}{l}
P_{1}-P_{n+1},i=j \\ 
-P_{n},\ \ \ \ \ \ \ \ \ i=j+1 \\ 
0,\ \ \ \ \ \ \ \ \ \ otherwise.%
\end{array}%
\right.
\end{equation*}%
Then $C^{-1}=(c_{ij}^{^{\prime }})$ is given by%
\begin{equation*}
c_{ij}^{\prime }=\left\{ 
\begin{array}{l}
\frac{P_{n}^{i-j}}{(P_{1}-P_{n+1})^{i-j+1}},i\geq j \\ 
0,\ \ \ \ \ \ \ \ \ \ otherwise.%
\end{array}%
\right.
\end{equation*}
\end{lemma}

\begin{proof}
Let $A:=(a_{ij})=CC^{-1}.$ Clearly, $a_{ij}=\tsum%
\nolimits_{k=1}^{n-2}c_{ik}c_{kj}^{^{\prime }}.$ When $i=j,$ we have%
\begin{equation*}
a_{ii}=(P_{1}-P_{n+1}).\frac{1}{P_{1}-P_{n+1}}=1.
\end{equation*}%
If $i>j,$ then%
\begin{eqnarray*}
a_{ij} &=&\tsum\nolimits_{k=1}^{n-2}c_{ik}c_{kj}^{^{\prime
}}=c_{i,i-1}c_{i-1,j}^{^{\prime }}+c_{ii}c_{ij}^{^{\prime }} \\
&=&-P_{n}\frac{P_{n}^{i-j-1}}{(P_{1}-P_{n+1})^{i-j}}+(P_{1}-P_{n+1})\frac{%
P_{n}^{i-j}}{(P_{1}-P_{n+1})^{i-j+1}}=0;
\end{eqnarray*}%
similar for $i<j.$ Thus, $CC^{-1}=I_{n-2}$.
\end{proof}

\begin{theorem}
\bigskip Let the matrix $\mathbb{P\ }$be $\mathbb{P}=$ $circ(P_{1},P_{2},%
\ldots ,P_{n})$ ($n\geq 3$). Then the inverse of the matrix $\mathbb{P}$ is%
\begin{equation*}
\mathbb{P}^{-1}=circ(p_{1},p_{2},\ldots ,p_{n})
\end{equation*}%
where%
\begin{eqnarray*}
p_{1} &=&\frac{1}{g_{n}}\left( 1+\frac{2P_{n}^{n-3}}{(P_{1}-P_{n+1})^{n-2}}%
+\dsum\limits_{k=1}^{n-3}\frac{P_{n-k}P_{n}^{k-1}}{(P_{1}-P_{n+1})^{k}}%
\right) \\
p_{2} &=&\frac{1}{g_{n}}\left( -2+\dsum\limits_{k=1}^{n-2}\frac{%
P_{n-k-1}P_{n}^{k-1}}{(P_{1}-P_{n+1})^{k}}\right) \\
p_{i} &=&-\frac{P_{n}^{i-3}}{g_{n}(P_{1}-P_{n+1})^{i-2}},\ \ \ i=3,4,\ldots
,n
\end{eqnarray*}%
for $g_{n}=P_{1}-2P_{n}+\tsum\nolimits_{k=2}^{n-1}P_{k-1}\left( \frac{P_{n}}{%
P_{1}-P_{n+1}}\right) ^{n-k}.$
\end{theorem}

\begin{proof}
Let%
\begin{equation*}
U=\left[ 
\begin{array}{cccccc}
1 & -g_{n}^{\prime } & \frac{g_{n}^{\prime }}{g_{n}}P_{n-2}-P_{n-1} & \frac{%
g_{n}^{\prime }}{g_{n}}P_{n-3}-P_{n-2} & \ldots & \frac{g_{n}^{\prime }}{%
g_{n}}P_{1}-P_{2} \\ 
0 & 1 & -\frac{P_{n-2}}{g_{n}} & -\frac{P_{n-3}}{g_{n}} & \ldots & -\frac{%
P_{1}}{g_{n}} \\ 
0 & 0 & 1 & 0 & \ldots & 0 \\ 
0 & 0 & 0 & 1 & \ldots & 0 \\ 
\vdots & \vdots & \vdots & \vdots & \ddots & \vdots \\ 
0 & 0 & 0 & 0 & \ldots & 0 \\ 
0 & 0 & 0 & 0 & \ldots & 1%
\end{array}%
\right]
\end{equation*}%
and $H=diag(1,g_{n})$ where $g_{n}^{\prime
}=P_{n}+\tsum\nolimits_{k=2}^{n-1}P_{k}\left( \frac{P_{n}}{1-P_{n+1}}\right)
^{n-k}$ and $g_{n}=P_{1}-2P_{n}+\tsum\nolimits_{k=2}^{n-1}P_{k-1}\left( 
\frac{P_{n}}{P_{1}-P_{n+1}}\right) ^{n-k}.$ Then we can write%
\begin{equation*}
M\mathbb{P}NU=H\oplus C
\end{equation*}%
where $H\oplus C$ is the direct sum of the matrices $H$ and $C$. Let $T=NU.$
Then we have%
\begin{equation*}
\mathbb{P}^{-1}=T(H^{-1}\oplus C^{-1})M.
\end{equation*}

Since the matrix $\mathbb{P}$ is circulant, its inverse is circulant from
Lemma 1.1 [1, p.9791]. Let%
\begin{equation*}
\mathbb{P}^{-1}=circ(p_{1},p_{2},\ldots ,p_{n}).
\end{equation*}%
Since the last row of the matrix $T$ is 
\begin{equation*}
\left( 0,1,-\frac{P_{n-2}}{g_{n}},-\frac{P_{n-3}}{g_{n}},-\frac{P_{n-4}}{%
g_{n}},\ldots ,-\frac{P_{2}}{g_{n}},-\frac{P_{1}}{g_{n}}\right) ,
\end{equation*}%
the last row components of the matrix $\mathbb{P}^{-1}$ are%
\begin{eqnarray*}
p_{2} &=&\frac{1}{g_{n}}\left( -2+\dsum\limits_{k=1}^{n-2}\frac{%
P_{n-k-1}P_{n}^{k-1}}{(P_{1}-P_{n+1})^{k}}\right) \\
p_{3} &=&-\frac{1}{g_{n}(P_{1}-P_{n+1})} \\
p_{4} &=&-\frac{P_{n}}{g_{n}(P_{1}-P_{n+1})^{2}} \\
p_{5} &=&-\frac{1}{g_{n}}\left( \dsum\limits_{k=1}^{3}\frac{%
P_{4-k}P_{n}^{k-1}}{(P_{1}-P_{n+1})^{k}}-2\dsum\limits_{k=1}^{2}\frac{%
P_{3-k}P_{n}^{k-1}}{(P_{1}-P_{n+1})^{k}}-\frac{P_{1}}{P_{1}-P_{n+1}}\right)
\\
&&\vdots \\
p_{n} &=&-\frac{1}{g_{n}}\left( \dsum\limits_{k=1}^{n-2}\frac{%
P_{n-k-1}P_{n}^{k-1}}{(P_{1}-P_{n+1})^{k}}-2\dsum\limits_{k=1}^{n-3}\frac{%
P_{n-k-2}P_{n}^{k-1}}{(P_{1}-P_{n+1})^{k}}-\dsum\limits_{k=1}^{n-4}\frac{%
P_{n-k-3}P_{n}^{k-1}}{(P_{1}-P_{n+1})^{k}}\right) \\
p_{1} &=&\frac{1}{g_{n}}\left( 1+\frac{2P_{n}^{n-3}}{(P_{1}-P_{n+1})^{n-2}}%
+\dsum\limits_{k=1}^{n-3}\frac{P_{n-k}P_{n}^{k-1}}{(P_{1}-P_{n+1})^{k}}%
\right) .
\end{eqnarray*}%
where $g_{n}=P_{1}-2P_{n}+\tsum\nolimits_{k=2}^{n-1}P_{k-1}\left( \frac{P_{n}%
}{P_{1}-P_{n+1}}\right) ^{n-k}.$ If $S_{n}^{(r)}=\tsum\nolimits_{k=1}^{r}%
\frac{P_{r-k+1}P_{n}^{k-1}}{(P_{1}-P_{n+1})^{k}}$ ($r=1,2,\ldots ,n-2$),
then we obtain%
\begin{equation*}
S_{n}^{(2)}-2S_{n}^{(1)}=\frac{P_{n}}{(P_{1}-P_{n+1})^{2}}
\end{equation*}%
and%
\begin{equation*}
S_{n}^{(r+2)}-2S_{n}^{(r+1)}-S_{n}^{(r)}=\frac{P_{n}^{r+1}}{%
(P_{1}-P_{n+1})^{r+2}},\ \ r=1,2,\ldots ,n-4.
\end{equation*}%
Hence, we have%
\begin{eqnarray*}
\mathbb{P}^{-1} &=&\frac{1}{g_{n}}circ\left(
1+2S_{n}^{(n-2)}+S_{n}^{(n-3)},-2+S_{n}^{(n-2)},-S_{n}^{(1)},-S_{n}^{(2)}+2S_{n}^{(1)},\right.
\\
&&\left. S_{n}^{(3)}-2S_{n}^{(2)}-S_{n}^{(1)},\ldots
,S_{n}^{(n-2)}-2S_{n}^{(n-3)}-S_{n}^{(n-4)}\right) \\
&=&\frac{1}{g_{n}}circ\left( 1+\frac{2P_{n}^{n-3}}{(P_{1}-P_{n+1})^{n-2}}%
+\dsum\limits_{k=1}^{n-3}\frac{P_{n-k}P_{n}^{k-1}}{(P_{1}-P_{n+1})^{k}}%
,-2+\dsum\limits_{k=1}^{n-2}\frac{P_{n-k-1}P_{n}^{k-1}}{(P_{1}-P_{n+1})^{k}}%
,\right. \\
&&\left. -\frac{1}{(P_{1}-P_{n+1})},-\frac{P_{n}}{(P_{1}-P_{n+1})^{2}},-%
\frac{P_{n}^{2}}{(P_{1}-P_{n+1})^{3}},\ldots ,-\frac{P_{n}^{n-3}}{%
(P_{1}-P_{n+1})^{n-2}}\right) .
\end{eqnarray*}%
by the nice properties of Pell numbers.
\end{proof}

{A Hankel matrix $A=(a_{ij})$ is an $n\times n$ matrix such that $%
a_{i,j}=a_{i-1,j+1}$. It is closely related to the Toeplitz matrix in the
sense that a Hankel matrix is an upside-down Toeplitz matrix.}

\begin{corollary}
\bigskip Let the matrix $M$ be as in (\ref{5}). Then its inverse is%
\begin{equation*}
M^{-1}=\left[ 
\begin{array}{rr}
1 & 0 \\ 
B & C%
\end{array}%
\right]
\end{equation*}%
where the matrices 
\begin{equation*}
B=\left( 
\begin{array}{rrrrr}
P_{n} & P_{n-1} & P_{n-2} & \ldots & P_{2}%
\end{array}%
\right) _{(n-1)\times 1}^{T}
\end{equation*}%
and $C$ is an $(n-1)\times (n-1)$ Hankel matrix that it has the first row $%
\left[ P_{n-1},P_{n-2},\ldots ,P_{1}\right] $ and the last column $\left[
P_{1},0,\ldots ,0\right] ^{T}.$\bigskip
\end{corollary}

\begin{proof}
The matrix $M^{-1}$ is obtained easily by applying elementary row\linebreak
operations to the augmented matrix $[M:I_{n}].$
\end{proof}

\begin{theorem}
\bigskip The matrix $\mathbb{Q}=circ(Q_{1},Q_{2},\ldots ,Q_{n})$ is
invertible when $n\geq 3$.
\end{theorem}

\begin{proof}
We show that $\det (\mathbb{Q})=2464\neq 0$ and $\det (\mathbb{Q}%
)=-1247232\neq 0$ by Theorem 2 for $n=3,4$. Then the matrix $\mathbb{Q}$ is
the invertible matrix for $n=3,4$. Let $n\geq 5.$ The Binet formula for
Jacobsthal-Lucas numbers yields $Q_{n}=\alpha ^{n}+\beta ^{n},$ where $%
\alpha +\beta =2$ and $\alpha \beta =-1.$ Then we have%
\begin{eqnarray*}
v(\omega ^{k}) &=&\dsum\limits_{r=1}^{n}Q_{r}\omega
^{kr-k}=\dsum\limits_{r=1}^{n}\left( \alpha ^{r}+\beta ^{r}\right) \omega
^{kr-k} \\
&=&\frac{\alpha (1-\alpha ^{n})}{1-\alpha \omega ^{k}}+\frac{\beta (1-\beta
^{n})}{1-\beta \omega ^{k}},\ \ (1-\alpha \omega ^{k},1-\beta \omega
^{k}\neq 0) \\
&=&\left( \frac{(\alpha +\beta )-(\alpha ^{n+1}+\beta ^{n+1})+\alpha \beta
\omega ^{k}(\alpha ^{n}+\beta ^{n})-2\alpha \beta \omega ^{k}}{1-(\alpha
+\beta )\omega ^{k}+\alpha \beta \omega ^{2k}}\right) \\
&=&\frac{2-Q_{n+1}+(2-Q_{n})\omega ^{k}}{1-2\omega ^{k}-\omega ^{2k}},\ \
k=1,2,\ldots ,n-1.
\end{eqnarray*}%
We are going to show that there is no $\omega ^{k},$ $k=1,2,\ldots ,n-1$
such that $v(\omega ^{k})=0.$ If $2-Q_{n+1}+(2-Q_{n})\omega ^{k}=0$ for $%
1-2\omega ^{k}-\omega ^{2k}\neq 0,$ then $\omega ^{k}=\frac{Q_{n+1}-2}{%
2-Q_{n}}$ would be a real number. By (\ref{8}) we would have $\sin \left( 
\frac{2k\pi }{n}\right) =0$ so that $\omega ^{k}=-1$ for $0<\frac{2k\pi }{n}%
<2\pi .$ However $x=-1$ is not a root of the equation $%
1-Q_{n+1}+(2-Q_{n})x=0 $ $(n\geq 5),$ a contradiction. i.e. $v(\omega
^{k})\neq 0$ for any $\omega ^{k},$ where $k=1,2,\ldots ,n-1\ $and $n\geq 5.$
Thus, the proof is completed by [1, Lemma 1.1].
\end{proof}

\begin{lemma}
If the matrix $A=\bigskip (a_{ij})_{i,j=1}^{n-2}$ is of the form%
\begin{equation*}
a_{ij}=\left\{ 
\begin{array}{l}
Q_{1}-Q_{n+1},\ i=j \\ 
2-Q_{n},\ \ \ \ \ \ i=j+1 \\ 
0,\ \ \ \ \ \ \ \ \ \ otherwise,%
\end{array}%
\right.
\end{equation*}%
then $A^{-1}=(a_{ij}^{^{\prime }})_{i,j=1}^{n-2}$ is given by%
\begin{equation*}
a_{ij}^{\prime }=\left\{ 
\begin{array}{l}
\frac{(Q_{n}-2)^{i-j}}{(Q_{1}-Q_{n+1})^{i-j+1}},\ i\geq j \\ 
0,\ \ \ \ \ \ \ \ \ \ \ \ \ otherwise.%
\end{array}%
\right.
\end{equation*}
\end{lemma}

\begin{proof}
Let $B:=(b_{ij})=AA^{-1}$ so that $b_{ij}=\tsum%
\nolimits_{k=1}^{n-2}a_{ik}a_{kj}^{^{\prime }}.$ Clearly%
\begin{equation*}
b_{ii}=(Q_{1}-Q_{n+1}).\frac{1}{Q_{1}-Q_{n+1}}=1.
\end{equation*}%
If $i>j,$ then%
\begin{eqnarray*}
b_{ij} &=&\tsum\nolimits_{k=1}^{n-2}a_{ik}a_{kj}^{^{\prime
}}=a_{i,i-1}a_{i-1,j}^{^{\prime }}+a_{ii}a_{ij}^{^{\prime }} \\
&=&(2-Q_{n})\frac{(Q_{n}-2)^{i-j-1}}{(Q_{1}-Q_{n+1})^{i-j}}+(Q_{1}-Q_{n+1})%
\frac{(Q_{n}-2)^{i-j}}{(Q_{1}-Q_{n+1})^{i-j+1}}=0;
\end{eqnarray*}%
similar for $i<j.$ Thus $A^{-1}A=I_{n-2}.$
\end{proof}

\begin{theorem}
\bigskip Let $n\geq 3.$ The inverse of the matrix $\mathbb{Q}$ is%
\begin{equation*}
\mathbb{Q}^{-1}=circ(q_{1},q_{2},\ldots ,q_{n})
\end{equation*}%
where%
\begin{eqnarray*}
q_{1} &=&\frac{1}{u_{n}}\left( 1-\frac{8(Q_{n}-2)^{n-3}}{%
(Q_{1}-Q_{n+1})^{n-2}}+\dsum\limits_{k=1}^{n-3}\frac{%
(Q_{n-k+2}-3Q_{n-k+1})(Q_{n}-2)^{k-1}}{(Q_{1}-Q_{n+1})^{k}}\right) \\
q_{2} &=&\frac{1}{u_{n}}\left(
-3+\dsum\limits_{k=1}^{n-2}(Q_{n-k+1}-3Q_{n-k})\left( \frac{(Q_{n}-2)^{k-1}}{%
(Q_{1}-Q_{n+1})^{k}}\right) \right) \\
q_{m} &=&\frac{4(Q_{n}-2)^{m-3}}{u_{n}(Q_{1}-Q_{n+1})^{m-2}},\ \ \
m=3,4,\ldots ,n
\end{eqnarray*}%
for $u_{n}=Q_{1}-3Q_{n}+\dsum\limits_{k=2}^{n-1}\left( Q_{k+1}-3Q_{k}\right)
\left( \frac{Q_{n}-2}{Q_{1}-Q_{n+1}}\right) ^{n-k}.$
\end{theorem}

\begin{proof}
Let%
\begin{eqnarray*}
\mathbb{S} &=&\left[ 
\begin{array}{cccc}
1 & -\frac{1}{2}u_{n}^{\prime } & \frac{u_{n}^{\prime }(Q_{n}-3Q_{n-1})}{%
2u_{n}}-\frac{Q_{n-1}}{2} & \frac{u_{n}^{\prime }(Q_{n-1}-3Q_{n-2})}{2u_{n}}-%
\frac{Q_{n-2}}{2} \\ 
0 & 1 & -\frac{Q_{n}-3Q_{n-1}}{u_{n}} & -\frac{Q_{n-}-3Q_{n-2}}{u_{n}} \\ 
0 & 0 & 1 & 0 \\ 
0 & 0 & 0 & 1 \\ 
\vdots & \vdots & \vdots & \vdots \\ 
0 & 0 & 0 & 0 \\ 
0 & 0 & 0 & 0%
\end{array}%
\right. \\
&&\left. ~\ \ \ \ \ \ \ \ \ \ \ \ \ \ \ \ \ \ \ \ \ \ \ \ \ \ \ \ \ \ \ \ \
\ \ \ \ \ \ \ \ 
\begin{array}{cc}
\ldots & \frac{u_{n}^{\prime }(Q_{3}-3Q_{2})}{2u_{n}}-\frac{Q_{2}}{2} \\ 
\ldots & -\frac{Q_{3}-3Q_{2}}{u_{n}} \\ 
\ldots & 0 \\ 
\ldots & 0 \\ 
\ddots & \vdots \\ 
\ldots & 0 \\ 
\ldots & 1%
\end{array}%
\right]
\end{eqnarray*}%
and $G=diag(2,u_{n})$ where $u_{n}=Q_{1}-3Q_{n}+\dsum\limits_{k=2}^{n-1}%
\left( Q_{k+1}-3Q_{k}\right) \left( \frac{Q_{n}-2}{Q_{1}-Q_{n+1}}\right)
^{n-k}$ and $u_{n}^{\prime }=\dsum\limits_{k=2}^{n}Q_{k}\left( \frac{Q_{n}-2%
}{Q_{1}-Q_{n+1}}\right) ^{n-k}.$ Then we obtain%
\begin{equation*}
K\mathbb{Q}L\mathbb{S}=G\oplus A
\end{equation*}%
where $G\oplus A$ is the direct sum of the matrices $G$ and $A$. If $W=L%
\mathbb{S},$ then we have%
\begin{equation*}
\mathbb{Q}^{-1}=W(G^{-1}\oplus A^{-1})K.
\end{equation*}

Since the matrix $\mathbb{Q}$ is circulant, the inverse matrix $\mathbb{Q}%
^{-1}$\ is circulant from Lemma 1.1 [1, p. 9791]. Let%
\begin{equation*}
\mathbb{Q}^{-1}=circ(q_{1},q_{2},\ldots ,q_{n}).
\end{equation*}%
Since the last row of the matrix $W$ is 
\begin{equation*}
\left( 0,1,-\frac{Q_{n}-3Q_{n-1}}{u_{n}},-\frac{Q_{n-1}-3Q_{n-2}}{u_{n}},-%
\frac{Q_{n-2}-3Q_{n-3}}{u_{n}},\ldots ,-\frac{Q_{4}-3Q_{3}}{u_{n}},-\frac{%
Q_{3}-3Q_{2}}{u_{n}}\right) ,
\end{equation*}%
the last row elements of the matrix $\mathbb{Q}^{-1}$ are%
\begin{eqnarray*}
q_{2} &=&\frac{1}{u_{n}}\left(
-3+\dsum\limits_{k=1}^{n-2}(Q_{n-k+1}-3Q_{n-k})\left( \frac{(Q_{n}-2)^{k-1}}{%
(Q_{1}-Q_{n+1})^{k}}\right) \right) \\
q_{3} &=&-\frac{Q_{3}-3Q_{2}}{u_{n}(Q_{1}-Q_{n+1})} \\
q_{4} &=&\frac{2(Q_{3}-3Q_{2})}{u_{n}(Q_{1}-Q_{n+1})}-\frac{1}{u_{n}}%
\dsum\limits_{k=1}^{2}\frac{(Q_{5-k}-3Q_{4-k})(Q_{n}-2)^{k-1}}{%
(Q_{1}-Q_{n+1})^{k}} \\
q_{5} &=&\frac{1}{u_{n}}\left( \frac{Q_{3}-3Q_{2}}{(Q_{1}-Q_{n+1})}%
-\dsum\limits_{k=1}^{3}\frac{(Q_{6-k}-3Q_{5-k})(Q_{n}-2)^{k-1}}{%
(Q_{1}-Q_{n+1})^{k}}\right. \\
&&\left. +2\dsum\limits_{k=1}^{2}\frac{(Q_{5-k}-3Q_{4-k})(Q_{n}-2)^{k-1}}{%
(Q_{1}-Q_{n+1})^{k}}\right) \\
&&\vdots \\
q_{n} &=&-\frac{1}{u_{n}}\left( \dsum\limits_{k=1}^{n-2}\frac{%
(Q_{n-k+1}-3Q_{n-k})(Q_{n}-2)^{k-1}}{(Q_{1}-Q_{n+1})^{k}}\right. \\
&&-2\dsum\limits_{k=1}^{n-3}\frac{(Q_{n-k}-3Q_{n-k-1})(Q_{n}-2)^{k-1}}{%
(Q_{1}-Q_{n+1})^{k}} \\
&&\left. -\dsum\limits_{k=1}^{n-4}\frac{(Q_{n-k-1}-3Q_{n-k-2})(Q_{n}-2)^{k-1}%
}{(Q_{1}-Q_{n+1})^{k}}\right) \\
q_{1} &=&\frac{1}{u_{n}}\left( 1-\frac{8(Q_{n}-2)^{n-3}}{%
(Q_{1}-Q_{n+1})^{n-2}}+\dsum\limits_{k=1}^{n-3}\frac{%
(Q_{n-k+2}-3Q_{n-k+1})(Q_{n}-2)^{k-1}}{(Q_{1}-Q_{n+1})^{k}}\right)
\end{eqnarray*}%
where $u_{n}=Q_{1}-3Q_{n}+\dsum\limits_{k=2}^{n-1}\left(
Q_{k+1}-3Q_{k}\right) \left( \frac{Q_{n}-2}{Q_{1}-Q_{n+1}}\right) ^{n-k}.$
Now, let us rearrange the $q_{i}$'s$\ (i\geq 4).$ Hence we obtain%
\begin{eqnarray*}
q_{4} &=&\frac{1}{u_{n}}\left( \frac{2Q_{3}-6Q_{2}-Q_{4}+3Q_{3}}{%
Q_{1}-Q_{n+1}}-\frac{(Q_{3}-3Q_{2})(Q_{n}-2)}{(Q_{1}-Q_{n+1})^{2}}\right) \\
&=&\frac{1}{u_{n}}\left( \frac{\overset{0}{\overbrace{-Q_{4}+2Q_{3}+Q_{2}}}+%
\overset{0}{\overbrace{3Q_{1}-Q_{2}}}}{Q_{1}-Q_{n+1}}-\frac{-4(Q_{n}-2)}{%
(Q_{1}-Q_{n+1})^{2}}\right) \\
&=&\frac{4(Q_{n}-2)}{u_{n}(Q_{1}-Q_{n+1})^{2}}
\end{eqnarray*}%
\begin{eqnarray*}
q_{5} &=&\frac{1}{u_{n}}\left[ \frac{1}{Q_{1}-Q_{n+1}}(\underset{0}{%
\underbrace{(-Q_{5}+2Q_{4}+Q_{3})}}+3(\underset{0}{\underbrace{%
Q_{4}-2Q_{3}-Q_{2}}}))\right. \\
&&\left. +\frac{Q_{n}-2}{(Q_{1}-Q_{n+1})^{2}}((\underset{0}{\underbrace{%
-Q_{4}+2Q_{3}+Q_{2}}})+\underset{0}{\underbrace{3Q_{1}-Q_{2}}})+\frac{%
4(Q_{n}-2)^{2}}{(Q_{1}-Q_{n+1})^{3}}\right] \\
&=&\frac{4(Q_{n}-2)^{2}}{(Q_{1}-Q_{n+1})^{3}}.
\end{eqnarray*}%
If we formulate general case $i\geq 3,$ then we have%
\begin{equation*}
q_{i}=\frac{4(Q_{n}-2)^{n-i}}{(Q_{1}-Q_{n+1})^{n-i}}.
\end{equation*}%
Thus%
\begin{eqnarray*}
\mathbb{Q}^{-1} &=&\frac{1}{u_{n}}circ\left( 1-\frac{8(Q_{n}-2)^{n-3}}{%
(Q_{1}-Q_{n+1})^{n-2}}+\dsum\limits_{k=1}^{n-3}\frac{%
(Q_{n-k+2}-3Q_{n-k+1})(Q_{n}-2)^{k-1}}{(Q_{1}-Q_{n+1})^{k}},\right. \\
&&-3+\dsum\limits_{k=1}^{n-2}(Q_{n-k+1}-3Q_{n-k})\left( \frac{(Q_{n}-2)^{k-1}%
}{(Q_{1}-Q_{n+1})^{k}}\right) ,\frac{4}{Q_{1}-Q_{n+1}}, \\
&&\left. \frac{4(Q_{n}-2)}{(Q_{1}-Q_{n+1})^{2}},\frac{4(Q_{n}-2)^{2}}{%
(Q_{1}-Q_{n+1})^{3}},\ldots ,\frac{4(Q_{n}-2)^{n-3}}{(Q_{1}-Q_{n+1})^{n-2}}%
\right)
\end{eqnarray*}%
by the nice properties fo Pell-Lucas numbers.
\end{proof}

\begin{corollary}
\bigskip Let the matrix $K$ be as in (\ref{7}). Then its inverse is block
matrix%
\begin{equation*}
K^{-1}=\left[ 
\begin{array}{rr}
1 & 0 \\ 
D & E%
\end{array}%
\right] .
\end{equation*}%
where the matrices 
\begin{equation*}
D=\left( 
\begin{array}{rrrrr}
\frac{Q_{n}}{2} & \frac{Q_{n-1}}{2} & \frac{Q_{n-2}}{2} & \ldots & \frac{%
Q_{2}}{2}%
\end{array}%
\right) _{(n-1)\times 1}^{T}
\end{equation*}%
and $E$ is an $(n-1)\times (n-1)$ Hankel matrix that it has the first row $%
\left[ P_{n-1},P_{n-2},\ldots ,P_{1}\right] $ and the last column $\left[
P_{1},0,\ldots ,0\right] ^{T}.$\bigskip
\end{corollary}

\begin{proof}
The matrix $K^{-1}$ is obtained easily by applying elementary row\linebreak
operations to the augmented matrix $[K:I_{n}].$
\end{proof}


\begin{thebibliography}{99}
\bibitem{1} S. Q. Shen, J. M. Cen and Y. Hao, On the determinants and
inverses of circulant matrices with Fibonacci and Lucas numbers, Appl. Math.
Comput. 217 (2011) 9790-- 9797.

\bibitem{2} S. Q. Shen, On the bounds for the norms of $r$-circulant
matrices with\linebreak Fibonacci and Lucas numbers, Appl. Math. Comput. 216
(2011) 2891-- 2897.

\bibitem{3} I. J. Good, On the inversion of circulant matrices, Biometrica,
37 (1950) 185-186.

\bibitem{4} P. J. Davis, Circulant Matrices, Wiley, NewYork, 1979.

\bibitem{5} M. Akbulak and D. Bozkurt, On the norms of Toeplitz matrices
involving Fibonacci and Lucas numbers, Hacettepe J. Math. and Stat. 37 (2)
(2008) 89-95.

\bibitem{6} M. Miladnovic \& P. Stanimirovic, Singular case of generalized
Fibonacci and Lucas numbers, J. Koeran Math. Soc. 48 (2011) 33-48.

\bibitem{7} G. Y. Lee, J. S. Kim and S. G. Lee, Factorizations and
Eigenvalues of\linebreak Fibonacci and symmetric Fibonacci matrices,
Fibonacci Quarterly 40 (2002) 203-211.

\bibitem{8} N. L. Tsitsas, E. G. Alivizatos \& G. H. Kalogeropoulos, A
recursive\linebreak algorithm for the inversion of matrices with circulant
blocks, Applied Math. And Comp., 188 (2007) 877-894.

\bibitem{9} G. Zhao, The Improved Nonsingularity on the $r$-Circulant
Matrices in signal processing, Inter. Conf. On Computer Techo. and
Development-ICCTD 2009, Kota Kinabalu, 564-567.

\bibitem{10} W. Zhao, The Inverse Problem of Anti-circulant Matrices in
Signal Processing, Pacific-Asia Conf. on Knowledge Engineering and Software
Engineering-KESE 2009, Shenzhen, 47-50.

\bibitem{11} S. Hal\i c\i\ and A. Da\c{s}demir, On some relationships among
Pell, Pell-Lucas and modified Pell sequences, SA\"{U}. Fen Bilimleri
Dergisi, 14 (2) (2010), 141-145.

\bibitem{12} D. Serre, Matrices: Theory and Applications, Springer, New
York, 2002.

\bibitem{13} R. A. Horn and C. R. Johnson, Matrix Analysis, Cambridge Univ.
Press, Cambridge, 1985.

\bibitem{14} R. Aldrovandi, Special Matrices of Mathematical
Physics:Stochastic, Circulant and Bell Matrices, World Scientific,
Singapore, 2001.
\end{thebibliography}
\end{document}